\documentclass[11pt]{amsart}
\usepackage{amssymb,mathrsfs,graphicx}

\title[From particle to kinetic and hydrodynamic descriptions of flocking]{From particle to kinetic and hydrodynamic descriptions of flocking}

\author[Seung-Yeal Ha]{Seung-Yeal Ha}
\address[Seung-Yeal Ha]{\newline Department of Mathematical Sciences\newline Seoul National University, Seoul 151-747, Korea}
\email{syha@snu.ac.kr}

\author[Eitan Tadmor]{Eitan Tadmor}
\address[Eitan Tadmor]{\newline
        Department of Mathematics,  Institute for Physical Science and Technology\newline
        and \newline
	Center of Scientific Computation And Mathematical Modeling (CSCAMM)\newline
        University of Maryland, 
        College Park, MD 20742 USA}
 \email[]{tadmor@cscamm.umd.edu}
\urladdr{http://www.cscamm.umd.edu/\~{}tadmor}

\newtheorem{theorem}{Theorem}[section]
\newtheorem{lemma}{Lemma}[section]
\newtheorem{corollary}{Corollary}[section]
\newtheorem{proposition}{Proposition}[section]
\newtheorem{remark}{Remark}[section]
\newenvironment{remarks}{{\flushleft {\bf Remarks.}}}{}

\newcommand{\bbr}{\mathbb R}

\def\charf {\mbox{{\text 1}\kern-.30em {\text l}}}
\def\Iz{{\mathcal M}_0}    
\def\Jz{J_0} 
\def\fin{f_0} 

\def\Pro{L}
\newcommand{\Qvf}[1]{\Pro[{#1}]}
\newcommand{\Qvfi}[2]{\Pro_{#1}[{#2}]}
\def\Qf{\Pro[f]} 
\def\Col{Q(f,f)} 
\def\EX{X}
\def\EV{V}  

\def\lam{\lambda}  
\def\vareps{\frac{\lam}{N}} 
\def\twovareps{\frac{2\lam}{N}}

\def\rA{A} 
\def\varepsA{\vareps \rA}
\def\ra{A} 
\def\Nvareps{\lam}

\def\Pt{\Omega} 
\def\Rt{\Sigma} 

\def\nablax{\nabla_x}
\def\nablav{\nabla_v}
\def\vxi{v}
\def\vst{v_{*}}

\begin{document}

\date{\today}


\thanks{\textbf{Acknowledgment.} This research was carried out when S.-Y. Ha was visiting the Department of Mathematics, University of Maryland, College Park, and it is a great pleasure to thank the faculty of applied mathematics group and their hospitality.
E. T. is grateful to Michelle Tadmor for attracting our initial interest to the topic of this research. The work of S.-Y. Ha is partially supported by KOSEF R01-2006-000-10002-0.
The work of E. Tadmor is supported by NSF grants DMS07-07749, NSF FRG grant DMS07-57227 and ONR grant N00014-91-J-1076.}

\begin{abstract}
We discuss the  Cucker-Smale's (C-S) particle model for flocking, deriving precise conditions for flocking to occur when pairwise interactions are sufficiently strong long range. 
We then derive a Vlasov-type kinetic model for the C-S particle model and  prove it 
exhibits
time-asymptotic flocking behavior for arbitrary compactly supported initial data. Finally, we introduce a hydrodynamic description of flocking based on the C-S Vlasov-type kinetic model and prove flocking behavior \emph{without} closure of higher moments.
\end{abstract}
\maketitle
\centerline{\date}

\tableofcontents

%
%
\section{Introduction}\label{sec:intro}
\setcounter{equation}{0}
Collective self-driven motion of self-propelled particles such as flocking of birds and mobile agents,
schooling of fishes, swarming of bacteria,   appears in many context,
e.g., biological organism \cite{Ao, D-M1, D-M2, D-M3, P-V-G, Pa, Sha, T-B, V-C-B-C-S},
mobile network \cite{A-H, C-D-M-B-C, D-C-B-C, He} appears in many contexts of biological system, mobile and human network \cite{C-S1, C-S2}. The flocking dynamics of
self-propelled particles is important to understand the nature of the aforementioned self-propelled particles.
 The terminology {\it "flocking"} represents the phenomenon in which self-propelled individuals using only limited environmental information and simple rules,
organize into an ordered motion (see \cite{T-T}, 
and it was a subject of 
biologists \cite{Ao, Pa}.  The study of flocking mechanism
 based on mathematical models was first started from the work of Viscek et al \cite{V-C-B-C-S}, and was further motivated by the
 hydrodynamic approach \cite{T-T}. 

Our starting point is a \emph{particle description}, proposed recently by Cucker-Smale \cite{C-S1, C-S2}, as a new simple dynamical system to explain the emergency of flocking mechanism with birds, with language evolution in primitive societies etc. The Cucker-Smale's system is different from previous flocking models, e.g., \cite{V-C-B-C-S}, in the sense that the collisionless momentum transfer between particles, $\{(x_i(t),v_i(t))\}_{i=1}^N$, is done through a long-range bi-particle interaction potential, $r(x,y)=r(|x-y|)$ depending on the distance $|x-y|$,
\[
\frac{d}{dt}v_i(t) =
\vareps \sum_{1 \leq j \leq N} r(x_i(t), x_j(t)) \big(v_j(t) - v_i(t)\big).
\]
The Cucker-Smale's flocking system (in short C-S system) is reviewed in Section \ref{sec:CS}. Here we revisit the  formation of flocking in C-S dynamics in terms of the \emph{fluctuations} relative to the center of mass $x_c(t):=1/N\sum x_i(t)$. The dynamics of fluctuations makes  transparent the flocking dynamics. Our main result, summarized in theorem  \ref{thm:flk}, improves \cite{C-S1} for slowly decaying interaction potential,  
$r(|x-y|) \sim |x-y|^{-2\beta},  2\beta\leq 1$. It is shown that 
flocking emerges in the sense that the following two main features occur: 
(i) the diameter $\max|x_i(t)-x_j(t)|$ remains uniformly bounded thus defining the ``flock",  
and (ii) the ``flock" is traveling with a bulk mean velocity which is asymptotically particle-independent, $v_i(t) \approx v_c:=1/N\sum v_i(0)$.

When the number of particles is sufficiently large, it is not economical to keep track of the motion of each particle through the Cucker-Smale's system. Instead, one is forced to study the mean field limit of C-S system and we introduce a \emph{kinetic description} for flocking, in analogy with the Vlasov equation in plasma and astrophysics.
In Section \ref{sec:kinetic}, we present a Vlasov type mean field model, which is derived from the C-S system using the BBGKY hierarchy in statistical mechanics.
 The formal derivation, carried in Section \ref{sec:kinetic}, follows by taking the limit of an 
$N$-particle interacting system consisting of self-propelled particles governed by C-S flocking dynamics. 
To this end, let $f =f(x,v,t)$ denote the one-particle distribution function of such particles positioned at $(x,t) \in \bbr^d\times \bbr_+$ with
 a velocity $v \in \bbr^d$: the dynamics of the distribution function $f$ is determined by 
\[
 \partial_t f + v\cdot\nabla_x f + \lambda \nablav\cdot\Col  = 0,
\]
where $\lambda$ is a positive constant, and $\Col$ is the interaction term
\[
\Col (x,v,t) := \int_{\bbr^{2d}} r(x,y) (\vst - v) f(x,v,t)f(y,\vst,t) d\vst dy,
\]
dictated by a prescribed interparticle interaction kernel, $r=r(x,y)$.
We refer to  Degond and Motsch   \cite{D-M1, D-M2, D-M3} for  recent kinetic description of  Vicsek type model of flocking.
The dynamics of particle trajectories of the proposed kinetic description of flocking is analyzed in Section \ref{subsec:kinetic-trajec}; in Section \ref{subsec:kinetic-global} we prove the global existence of smooth solutions to the kinetic model with arbitrary smooth compactly supported initial data.
In Section \ref{sec:kinetic-time}, we show that the kinetic model reveals the time-asymptotic flocking behavior when the bounded interparticle interaction rate has a sufficiently strong long range. Our results are summarized in the main theorem \ref{thm:asy}, proving the decay of energy fluctuations, $\Lambda[f](t)$, around the mean bulk velocity, $u_c$,
\[
\Lambda[f](t):= \int_{\bbr^{2d}}|v-u_c|^2f(x,v,t)dvdx,
\qquad u_c(t)= \frac{\int_{\bbr^{2d}}vf(x,v,t)dvdx}{\int_{\bbr^{2d}} f(x,v,t)dvdx}\equiv u_c(0).
\] 
Flocking is proved for the restricted range, $2\beta <1/2$, realized by the  asymptotic decay estimate, $\Lambda[f](t) \rightarrow 0$ as $t\rightarrow \infty$.

In Section \ref{sec:hydro} we turn our attention to the \emph{hydrodynamic description} of flocking, furnished by moments of the kinetic distribution function. We study  the dynamics of the resulting system of balanced laws related to the moments of Vlasov model. Despite the lack of closure, we present a fundamental estimate  which enables to conclude the flocking mechanism at the macroscopic  hydrodynamic scales. Theorem \ref{thm:fundamental} states that the energy-related functional, $\Gamma(t)$
\[
\displaystyle {\Gamma}(t)  := \int_{\bbr^{2d}}\Big(\frac{1}{2}|u(x) - u(y)|^2 + e(x) + e(y) \Big) \rho(x) \rho(y) dy dx.
\]
decays provided the interparticle interaction, $\varphi(s)=\inf r(x(s),y(s))$ decays \emph{slowly} enough so that its primitive, $\Phi(t)$, diverges. This in turn in related to the
\emph{increase} of entropy
\[
\frac{d}{dt}\int_{\bbr^{2d}} f\log(f) dx dv \geq 0,
\]
as particles with increasingly highly correlated velocities flock towards particle-independent bulk velocity. 
%
%
\section{A particle description of flocking}\label{sec:CS}
\setcounter{equation}{0} 
\subsection{The  Cucker-Smale model}
In this section, we briefly review the Cucker-Smale's flocking system in \cite{C-S1, C-S2, She}, which
manifests the time-asymptotic flocking behavior of many particle systems.
We reinterpret the C-S system in terms of \emph{fluctuations} relative to the center of mass coordinates, which enables us to simplify and sharpen the derivation of sufficient conditions for flocking to occur.

Consider an $N$-particle interacting system consisting of identical
particles with mass $m$ to be assumed to be unity.  Let $[x_i(t),
v_i(t)] \in \bbr^{2d}$ be the phase space position of an
$i$-particle.  
The Cucker-Smale dynamical system \cite{C-S1, C-S2} takes the form
\begin{equation} \label{B1}
\displaystyle \frac{d}{dt}x_i(t) = v_i(t), \qquad \frac{d}{dt}v_i(t) =
\vareps \sum_{1 \leq j \leq N} r(x_i(t), x_j(t)) \big(v_j(t) - v_i(t)\big).
\end{equation}
Here $\lam$ is a positive constant, and $r(x,y)$ is a symmetric, bi-particle \emph{interaction kernel},
\begin{equation}\label{Br}
r(x,y) = r(y,x)\leq \rA.
\end{equation}
\begin{subequations}\label{eqs:Bd}
To discuss the time asymptotic flocking behavior, we will restrict our attention to interparticle interactions which are decreasing functions of the distance\footnote{To emphasize this point, we therefore continue to refer to general symmetric kernels, $r(x,y)=r(y,x)$, whenever translation invariance is not necessary.}, 
\begin{equation}\label{Bd}
r(x,y)=r(|x-y|), \qquad r(\cdot) \ \ \mathrm{is decreasing}.
\end{equation} 
A prototype example is the interaction kernel with a polynomial decay of order $2\beta$, \cite{C-S1},
\begin{equation} \label{B2}
 \displaystyle r(|x-y|) \geq  \frac{\ra}{(1 + |x-y|^2)^{\beta}}, \qquad \beta \geq 0.
\end{equation}
We note in passing that  only a lower-bound of the interaction kernel matters. 
\end{subequations}

For notational simplicity  we often omit $t$-dependence from the 
particle identification, abbreviating $x_i \equiv x_i(t)$ and $v_i \equiv v_i(t)$. 

Let $m_j(t), \ j=0,1,2$ denote the moments.
\[ 
\displaystyle  m_0 := \sum_{i=1}^{N} 1 = N, \qquad   m_1(t) := \sum_{i=1}^{N} v_i(t), \qquad   m_2(t) := \sum_{i=1}^{N} |v_i(t)|^2. 
\]
Regarding the dynamics of these moments, we have the following estimate.

\begin{proposition}\label{prop:xyz}
Let $(x_i(t), v_i(t))$ be the solution to the C-S system \eqref{B1},\eqref{Br}. 
Then the following estimates hold.
\begin{subequations}\label{eqs:B5}
\begin{eqnarray}
&& \frac{d}{dt} m_1(t) =0. \label{eq:B5a} \\ 
&& \frac{d}{dt} m_2(t)= -\vareps
\sum_{1 \leq i, j \leq N} r(x_i, x_j) |v_j - v_i|^2. \label{eq:B5b} \\
&& m_2(t) \geq m_2(0)
e^{-2\Nvareps \rA t} + \frac{|m_1(0)|^2}{m_0} \Big( 1 - e^{-2\Nvareps \rA
t}\Big). \label{eq:B5c}
\end{eqnarray}
\end{subequations}
\end{proposition}

\begin{remark}\label{rem:xyz} Proposition \ref{prop:xyz} tells us that although the kinetic energy $m_2(t)$ is monotonically decreasing, (\ref{eq:B5b}), it has the following nonzero lower bound if initial momentum $m_1(0) \not = 0$,
\begin{equation}\label{eq:CSinq}
m_2(t) \geq \frac{|m_1(0)|^2}{m_0}.
\end{equation}
Flocking occurs when equality takes place in the Cauchy-Schwarz inequality (\ref{eq:CSinq}).
\end{remark}

\begin{proof}
Conservation of momentum in \eqref{eq:B5a} follows from the symmetry $r(x_i, x_j) = r(x_j, x_i)$, for  
\[ 
\displaystyle \frac{d}{dt} \Big( \sum_{i=1}^{N} v_i \Big) = \vareps \sum_{1 \leq i, j \leq N}r(x_i, x_j) (v_j - v_i) = 0. 
\]
Moreover, symmetry also implies
\[
\displaystyle \sum_{1 \leq i, j \leq N} r(x_i, x_j) v_i \cdot (v_i - v_j) 
= -\sum_{1 \leq i, j \leq N} r(x_i, x_j) v_j \cdot (v_i - v_j)
\] 
and hence the energy dissipation \eqref{eq:B5b} follows
\begin{eqnarray*}
\displaystyle  \frac{d}{dt} \Big( \sum_{i=1}^{N} |v_i|^2 \Big) &=&
-\twovareps
 \sum_{1 \leq i, j \leq N} r(x_i,  x_j) v_i \cdot (v_i - v_j) \\
 &=&  \twovareps \sum_{1 \leq i, j \leq N} r(x_i,  x_j) v_j \cdot (v_i - v_j) 
 = -\vareps \sum_{1 \leq i, j \leq N}  r(x_i,  x_j) |v_i - v_j|^2. 
\end{eqnarray*}

Finally, to prove \eqref{eq:B5c}, we use the energy dissipation in \eqref{eq:B5b}, the fact that $r(x_i, x_j) \leq \rA$ and the conservation of momentum in \eqref{eq:B5a} to find
\begin{align*}
\begin{aligned}
\displaystyle \frac{d}{dt} m_2(t)
 &= -\vareps \sum_{1 \leq i, j \leq N} r(x_i, x_j) |v_i - v_j|^2
  \geq -\varepsA \sum_{1 \leq i, j \leq N}  |v_i - v_j|^2   \\ 
 &= -2\Nvareps\rA  \Big( m_2(t)-\frac{|m_1(t)|^2}{N} \Big)= 
-2\Nvareps \rA  \Big( m_2(t)-\frac{|m_1(0)|^2}{m_0} \Big).
\end{aligned}
\end{align*}
Gronwall's lemma yields \eqref{eq:B5c}.
\end{proof}

\subsection{Asymptotic behavior of fluctuations --- flocking}
We now turn to study the asymptotic time behavior of solutions to C-S system \eqref{B1},\eqref{Bd}. To this end, we introduce a center of mass system $(x_c(t), v_c(t))$,
\[ 
\displaystyle x_c(t) := \frac{1}{N} \sum_{i=1}^{N} x_i(t), \qquad v_c(t) := \frac{1}{N} \sum_{i=1}^{N} v_i(t). 
\]
Then, thanks to conservation of momentum, the velocity $v_c$ is constant in $t$, and the
trajectory of center of mass $x_c$ is  a straight line:
\[  
v_c(t) = v_c(0), \qquad  x_c(t) = x_c(0) + t v_c(0). 
\]
Observe that the fluctuations around the center of mass, 
\[
x_i(t) \mapsto x_i(t)-x_c(t), \qquad v_i(t) \mapsto v_i(t)-v_c(t),
\]
satisfy the same C-S system \eqref{B1},\eqref{Bd}: it is here that we take into account  the fact that the interparticle kernel depends on the distance, $r(x,y)=r(|x-y|)$. We shall show that under appropriate conditions, flocking occurs in the sense that these \emph{fluctuations} decay in time. Thus, the time-asymptotic dynamics of C-S solutions emerges as a linear movement with a fixed velocity dictated by the coordinates of  center of mass.
    
To proceed, we introduce the two auxiliary functions which measure the fluctuations of the fluctuations around their center of mass,
\begin{eqnarray*}
{\EX}(t) := \sum_{1\leq i \leq N} |x_i(t)
- x_c(t)|^2, \qquad 
{\EV}(t) := \sum_{1\leq i, j \leq N} |v_i(t) - v_c(0)|^2,
\end{eqnarray*}
subject to initial conditions $(\EX_0,\EV_0)=(\EX(0),\EV(0))$.
The flocking behavior will depend in an essential way on the behavior of the minimal value of the interparticle interaction at time $t$,
\begin{equation}
\varphi(t) := \min_{1\leq i,j \leq N} r(x_i(t),x_j(t)).
\end{equation}
We begin with the fluctuations of velocities.
\begin{lemma}\label{lem:2-1}[Fluctuations of velocities].
Let $(x_i(t), v_i(t))$ be the solution of the system \eqref{B1},\eqref{Bd}. Then we have
\[
\displaystyle  {\EV}(t) \leq  {\EV}_0 e^{-2 \Nvareps
 \Phi(t) },  \ \ \Phi(t):=\int_0^t \varphi(\tau) d\tau.
\]
\end{lemma}

\begin{proof}
We invoke  \eqref{eq:B5b} with $v_i(t)-v_c(t)$ replacing $v_i(t)$ to find that
\begin{equation}\label{eq:2.4ca}
\frac{d}{dt} \sum_{i=1}^{N} |v_i- v_c|^2  =  -  \vareps\sum_{1\leq i, j \leq N} r(x_i, x_j) |v_j - v_i|^2.
\end{equation}
Since $\sum_{i} (v_i - v_c) = 0$, we have 
$\sum_{1\leq i, j \leq N} |v_i - v_j|^2 = 2 N \EV(t)$, 
and the result follows from Gronwall's integration of 
\begin{eqnarray*}
\displaystyle  \frac{d}{dt} \EV(t)  =  -  \vareps\sum_{1\leq i, j \leq N} r(x_i, x_j) |v_j - v_i|^2 \leq - 2 \Nvareps \varphi(t)  \EV(t).
\end{eqnarray*}
\end{proof}

\begin{remark}\label{rem:2-1}
Lemma \ref{lem:2-1} implies the sufficient condition for flocking is that the interparticle interaction potential decays \emph{sufficiently slow}, so that its primitive, $\Phi(t)$,  diverges, i.e.,  
\begin{equation}\label{eq:divergence}
\displaystyle  {\rm if } \ \  \lim_{t\rightarrow \infty}\Phi(t)\equiv \int_0^{t} \varphi(\tau) d\tau = \infty  \ \ \ {\rm then} \ \   \lim_{t \to \infty} |v_i(t) - v_c| = 0, \qquad i =1, \cdots, N. 
\end{equation}
\end{remark}
The  answer whether $\varphi(t)$ decays sufficiently slow to enforce flocking depends on the variance of positions $x_i(t)$.
\begin{lemma}\label{lem:2-2}[Fluctuations of positions].
Let $(x_i(t), v_i(t))$ be the solution of the system \eqref{B1},\eqref{Bd}. Then we have
\[
\displaystyle {\EX}(t) \leq  2 {\EX}_0 +
{\EV}_0\frac{ t^2}{2}, \qquad t \geq 0.
\]
\end{lemma}
\begin{proof}
We use Cauchy-Schwartz's inequality to see
\begin{eqnarray*}
\displaystyle \frac{d}{dt} \sum_{i=1}^{N} |x_i - x_c|^2 &=& \sum_{i=1}^{N}
(x_i - x_c) \cdot \left(\frac{dx_i}{dt}  - \frac{d x_c}{dt}\right) \cr &=&
\sum_{i=1}^{N} (x_i - x_c) \cdot (v_i - v_c) \leq 
\sqrt{\sum_{i=1}^{N}|x_i - x_c|^2} \sqrt{\sum_{i=1}^{N} |v_i - v_c|^2}.
\end{eqnarray*}
Using Lemma \ref{lem:2-1} we obtain,
$\displaystyle \frac{d}{dt} {\EX}(t) \leq \sqrt{{\EV}(t)} \sqrt{{\EX}(t)} 
\leq \sqrt{\EV_0} e^{-\Nvareps \Phi(t)} \sqrt{\EX(t)}$,
and the solution of this differential inequality  yields
\begin{equation}\label{eq:rough}
\displaystyle {\EX}(t) \leq 2 {\EX}_0 + \frac{{\EV}_0}{2}
\Big[ \int_0^t e^{- \Nvareps \Phi(\tau)}
d\tau \Big]^2 \leq 2 {\EX}_0 + {\EV}_0\frac{ t^2}{2}.
\end{equation}
\end{proof}

As a corollary of Lemma \ref{lem:2-2} we now obtain the desired lower bound for
$\varphi(t)$. 

\begin{corollary}\label{cor:2-1} Let $(x_i(t), v_i(t))$ be the solutions to 
\eqref{B1},\eqref{Bd}. Then $\varphi(t)$ satisfies
\[
\displaystyle  \varphi(t) \geq r\big(\sqrt{2\EX(t)}\big) \geq r\big(\sqrt{4{\EX}_0 + {\EV}_0 t^2}\big).
\]
\end{corollary}
\begin{proof} We use Lemma \ref{lem:2-2} to see that for each $i,j \in \{1, \cdots, N \}$,
\begin{eqnarray*}
\displaystyle |x_i - x_j|^2 \leq 2 (|x_i - x_c |^2 + |x_j -
x_c|^2) \leq  2 {\EX}(t) \leq  4{\EX}_0 + {\EV}_0 t^2.
\end{eqnarray*}
Since $r(\cdot)$ is decreasing, we have
$\displaystyle \varphi(t)  = \min_{1 \leq i, j \leq N} r(x_i(t), x_j(t))  
\geq r\big(\sqrt{4{\EX}_0 + {\EV}_0 t^2}\big)$.
\end{proof}

The asymptotic flocking now depends on the specific decay of the interparticle interaction $r(\cdot)$. As an example,  consider the C-S system with the $2\beta$ interaction \eqref{B1},\eqref{B2}, where

\begin{equation}\label{B-5}
\displaystyle  \varphi(t) \geq \ra(1 + 4{\EX}_0 + {\EV}_0 t^2)^{-\beta}  
\geq \ra\kappa_1 (1 + t)^{-2\beta}, \qquad \kappa_1:=\Big( \max\{ 1 + 4 {\EX}_0, {\EV}_0 \} \Big)^{-\beta}. 
\end{equation}
We conclude  that the divergence of $\Phi(t)=\int^t\varphi(\tau)d\tau$ and hence, by \eqref{eq:divergence} that flocking occurs, for $2\beta<1$. This recovers the Cucker-Smale result \cite{C-S1, C-S2}. Below we improve the Cucker-Smale result proving \emph{unconditional} flocking result for $\beta=1/2$. 

\begin{theorem}\label{thm:flk}
Let $(x_i(t), v_i(t))$ be the solutions to \eqref{B1},\eqref{B2}
with $\EV_0>0$. Then the following holds.

(i) There exist a positive constant, $C_2$ (depending only on $\kappa_1, \ra$ and $\beta$ as specified in \eqref{eqs:C2} below), such that 
\begin{subequations}
\begin{equation}\label{eq:expX}
|x_i(t)-x_c| \lesssim |x_i(0)-x_c| +C_2.
\end{equation}
 
(ii) There exists constants, $\kappa_i>0, i=1,2$, such that
\begin{equation}\label{eq:exp}
\displaystyle  |v_i(t) - v_c| \lesssim \sqrt{\EV_0}\times
\left\{
\begin{array}{ll}
     e^{-\Nvareps\ra\kappa_2 t}, & \beta \in [0, \frac{1}{2}), \qquad \kappa_2:={(1 + 4 {\EX}_0 + 8 C_2^2)^{-\beta}},\\ \\
     \displaystyle (1 + t)^{-\Nvareps\ra\kappa_1} ,&\beta = \frac{1}{2}, \qquad \kappa_1=\Big( \max\{ 1 + 4 {\EX}_0, {\EV}_0 \} \Big)^{-\beta}. \\
\end{array}
\right.
\end{equation}
\end{subequations}

\begin{remark}\label{rem:flk} Theorem  \ref{thm:flk} shows the two main features of flocking occur with the $2\beta$-interaction potential,  $2\beta\leq 1$, namely, the diameter 
$\max|x_i(t)-x_j(t)|$ remains uniformly bounded thus defining the traveling ``flock" with 
 velocity which is asymptotically particle-independent, $v_i(t) \approx v_c$.
\end{remark}

\end{theorem}
\begin{proof} We begin with the case $0 \leq \beta < \frac{1}{2}$. To get the optimal exponential convergence rate, we employ a bootstrapping argument in three steps.

\medskip\noindent
{\bf Step 1}.  We first obtain a weak integrable decay rate for $|v_i - v_c|$. Using \eqref{B-5} we find
\begin{subequations}\label{eqs:C2}
\begin{equation}\label{eq:C2a}
\displaystyle -\int_0^t \varphi(\tau) d\tau \leq -\ra\kappa_1 \int_0^t (1 +
\tau)^{-2\beta} d\tau \lesssim C_1\Big( 1 - (1 + t)^{1-2\beta} \Big), \qquad 0\leq \beta<\frac{1}{2}.
\end{equation}
The above estimate together with  Lemma \ref{lem:2-1} yield
${\EV}(t) \lesssim   {\EV}_0 e^{-2C_1(1 + t)^{1-2\beta}}$. 

\medskip\noindent
{\bf Step 2}. Next, we  improve lemma  \ref{lem:2-2}, observing that for $\beta<1/2$, the position $\EX(t)$ remains uniformly bounded in time. Indeed, we have,
\[
\displaystyle x_i(t) = x_i(0) + \int_0^{t} v_i(\tau) d\tau, \qquad 
\displaystyle x_c(t) = x_c(0) + \int_0^t v_c d\tau.
\]
Time integrability of $|v_i(t)-v_c(t)|$ then yields (\ref{eq:expX}),
\begin{eqnarray*}
\displaystyle |x_i(t) - x_c(t)| &\leq& |x_i(0) - x_c(0)| + \int_0^t
|v_i(\tau) - v_c| d\tau \cr \displaystyle
&\lesssim& |x_i(0) - x_c(0)| + \int_0^{\infty} e^{-C_1 (1 +
t)^{1-2\beta}} dt \leq |x_i(0) - x_c(0)| + C_2,
\end{eqnarray*}
where
\begin{equation}\label{eq:C2}
C_2 \lesssim \int_0^{\infty} e^{-C_1 (1 + t)^{1-2\beta}} dt <
\infty, \quad 0<\beta < \frac{1}{2}. 
\end{equation}
\end{subequations}
\noindent
{\bf Step 3}. The uniform bound of $\EX(t)$ implies an improved estimate for the interparticle interaction $\varphi(t)$: corollary \ref{cor:2-1} implies
\[ 
\displaystyle \varphi(t) \geq r(\sqrt{2\EX_0}) \geq {\ra}{(1 + 4 {\EX}_0 + 8 C_2^2 )^{-\beta}}=\ra\kappa_2,
\]
which in turn, using lemma \ref{lem:2-1}, yields the optimal exponential convergence rate
\eqref{eq:exp}
\[ 
\displaystyle |v_i(t)-v_c(t)|^2 < \EV(t) \leq \EV_0 e^{-2 \Nvareps
\Phi(t)} \lesssim  e^{-2\Nvareps \ra\kappa_2 t}.
\]

\bigskip
It remains to deal with the case $\beta =\frac{1}{2}$. Here, we have
\[
\displaystyle -2 \Nvareps \Phi(t)  \leq -2\Nvareps \ra \Big(
\max\{ 1 + 4 {\EX}_0, {\EV}_0 \}\Big)^{-\frac{1}{2}}
\int_0^t (1 + \tau)^{-1} d\tau  = -2\Nvareps \ra\kappa_1 \ln (1 + t),
\]
which  in turn implies \eqref{eq:exp}, ${\EV}(t) \leq  {\EV}_0  e^{-2\Nvareps \Phi(t)} \leq  {\EV}_0  (1 + t)^{-2\Nvareps\ra\kappa_1}$. 
\end{proof}

\begin{remarks}
\begin{enumerate}
\item Consider  the borderline case $\beta = \frac{1}{2}$ with initial configuration satisfying
\[ 
\displaystyle \Nvareps\kappa_1 > 1, \quad \mbox{ i.e.,} \quad \sqrt{\max\{ 1 + 4 {\EX}_0, 
{\EV}_0 \}} < \Nvareps; 
\]
then the same bootstrapping argument used for $\beta \in [0, \frac{1}{2})$ gives the exponential convergence:
\[ 
\displaystyle |v_i(t) - v_c | \leq \sqrt{\EV_0} e^{-{ \ra\widetilde\kappa}_2 t}. 
\]

\item Flocking occurs even if $\beta > \frac{1}{2}$, but only  for special initial configurations. Sufficient flocking conditions for such initial profiles is presented in \cite{C-S1}.
\end{enumerate}
\end{remarks}

\section{From particle to kinetic description of flocking}\label{sec:kinetic}
\setcounter{equation}{0} 
\subsection{Derivation of a mean-field model}

We assume that the number of particles involved in the C-S model \eqref{B1},\eqref{Br} is large enough that it becomes meaningful to observe the $N$-particle distribution function, 
\begin{subequations}\label{eqs:fN}
\begin{equation}\label{eq:fN}
f^N =f^N(x_1,v_1, \ldots, x_N,v_N,t), \quad  (x_i,v_i)\in \bbr^{d}\times \bbr^{d}.
\end{equation}
Since particles are indistinguishable, the probability density $f^N=f^N(\cdot)$ is symmetric in its phase-space arguments,
\begin{equation}\label{eq:fNsymm}
f^N( \cdots, x_i, v_i, \cdots x_j, v_j, \cdots,t) = f^N(\cdots, x_j,
v_j, \cdots x_i, v_i, \cdots,t),
\end{equation}
\end{subequations}
so we can `probe' $f^N$ by any of its $N$ pairs of phase-variables. 
Let $f^N(\cdot,\cdot,t)$ denote the marginal distribution
\[
f^N(x_1,v_1,t) := \int_{\bbr^{2d(N-1)}}f^N (x_1, v_1, x_-,v_-,t) dx_- dv_-, \qquad (x_-,v_-):=(x_2, v_2, \cdots, x_N, v_N).
\]

The formal derivation of a kinetic description for the C-S particle system \eqref{B1},\eqref{Br} is   carried out below  using  the 
BBGKY hierarchy, e.g., \cite{B-D-P, R-S1, R-S2},  based on  the Liouville equation,  \cite{K} 
\begin{equation} \label{eq2-1}
\displaystyle \partial_t f^N + \sum_{i=1}^{N} v_i\cdot \nabla_{x_i} f^N  +
 \vareps \sum_{i =1}^{N} \nabla_{v_i}\cdot \Big( \sum_{j =1}^{N} r(x_i, x_j) (v_j - v_i) f^N  \Big) = 0.
\end{equation}
To this end, one study the marginal distribution $f^N(x_1,v_1,t)$ by integration of \eqref{eq2-1} with respect to $dx_- dv_-=dv_2dx_2\cdots dv_Ndx_N$ (to simplify the notations, we now suppress the time-dependence whenever it is clear by the context, denoting $f^N(x_1, v_1, \cdots, x_N, v_N,t) = f^N(x_1, v_1, \cdots, x_N, v_N)$). Since $f^N(\cdot,\cdot)$ is rapidly decaying at infinity, the transport term  in \eqref{eq2-1} amounts to
\begin{equation}\label{eq:trans}
\int_{\bbr^{2d(N-1)}} \sum_{i=1}^{N}  v_i\cdot \nabla_{x_i} f^N 
dx_{-} dv_{-}  =  v_1\cdot\nabla_{x_1} f^N(x_1,v_1).
\end{equation}
The corresponding integration of the forcing term in \eqref{eq2-1}, yields
\begin{eqnarray*}
&& \vareps \sum_{i=1}^{N} \int_{\bbr^{2d(N-1)}} \sum_{j=1}^{N} \nabla_{v_i}\cdot \Big( r(x_i, x_j)(v_j - v_i) f^N  \Big) dx_-dv_-  \cr && \hspace{0.5cm} =~ \vareps
\int_{\bbr^{2d(N-1)}}\sum_{2\leq j \leq N}
\nabla_{v_1}\cdot \Big(r(x_1, x_j) (v_j - v_1) f^N  \Big) dx_-dv_-.
\end{eqnarray*}
But the symmetry of $f^N$, \eqref{eq:fNsymm}, implies that the integrals being summed above are the same for $j=2,3 \ldots, N$. Consequently, it will suffice to consider $j=2$:
\begin{eqnarray}
&& \vareps \sum_{i=1}^{N} \int_{\bbr^{2d(N-1)}} \sum_{j=1}^{N} \nabla_{v_i}\cdot \Big( r(x_i, x_j)(v_j - v_i) f^N  \Big) dx_{-}dv_{-}  \nonumber \\ && \hspace{0.5cm} =~
 \frac{\lam}{N}(N-1) \int_{\bbr^{2d(N-1)}} r(x_1, x_2) \nabla_{v_1}\cdot \Big( (v_2 - v_1) 
f^N  \Big)dx_{2} dv_{2} \cdots d x_N d v_N \label{eq:force} \\
  && \hspace{0.5cm}  = \left(\lam-\frac{\lam}{N}\right)  
\nabla_{v_1}\cdot \Big( \int_{\bbr^{2d}} r(x_1, x_2)(v_2 - v_1) g^N dx_{2} dv_{2} \Big).
\nonumber
\end{eqnarray}
Here $g^N$ is the two-particle marginal function
\[
g^N(x_1,v_1,x_2,v_2,t) := \int_{\bbr^{2d(N-2)}}f^N dx_{3} dv_{3} \cdots d x_N d v_N.
\] 
Thus, in view of \eqref{eq:trans} and \eqref{eq:force}, 
marginal integration of \eqref{eq2-1} over $(x_-,v_-)$ implies that  
the one-particle density function, $f^N(x_1,v_1,t)$, satisfies
\begin{eqnarray*}
\displaystyle  \partial_t f^N + v_1\cdot\nabla_{x_1} f^N 
               + \left(\lam-\frac{\lam}{N}\right) \nabla_{v_1}\cdot 
		\Big( \int_{\bbr^{2d}} r(x_1,x_2)(v_2 - v_1) g^N dx_{2} dv_{2} \Big)=0.
\end{eqnarray*}
We now take the mean-field limit $N\rightarrow \infty$: we end up with the  one- and two-particle limiting densities, $f:=\lim_{N\rightarrow \infty} f^N(x_1,v_1)$ and $g:=\lim_{N\rightarrow \infty} g^{N}(x_1,v_1,x_2,v_2)$, which satisfy 
\begin{equation}
\displaystyle \partial_t f +  v_1\cdot\nabla_{x_1}f +
 \lam  \nabla_{v_1}\cdot \Big( \int_{\bbr^{2d}} r(x_1, x_2)(v_2 - v_1) g dx_{2} dv_{2}=0 \Big).
\end{equation}
To close the above equation we make the ``molecular
chaos" assumption about the independence of the two-point particle distribution,
\[
\displaystyle g(x_1, v_1, x_2, v_2, t) = f(x_1, v_1, t) f(x_2,v_2,t);
\]
Relabel, $(x_1,v_1)\mapsto (x,v)$ and $(x_2,v_2)\mapsto (y,\vst)$. We conclude that the one-particle distribution function $f(x,v,t)$ satisfies the Vlasov-type mean-field model,

\begin{subequations}\label{D-1}
\begin{align}\label{eq:D-1a}
\displaystyle & \partial_t f +  v\cdot\nablax f + \lambda
\nablav\cdot\Col=0, \\ 
\displaystyle & \Col (x,v,t) :=
\int_{\bbr^{2d}} r(x,y) (\vst- v) f(x,v,t)f(y,\vst,t) d\vst dy.
\end{align}
Here, $\Col$ is the quadratic interaction which can be expressed in the equivalent form
\begin{align}\label{eq:D-1b}
\displaystyle & \Col (x,v,t) = f\Qf, \quad \Qf(x,v,t):= 
\int_{\bbr^{2d}} r(x,y) (\vst-v)f(y,\vst,t) d\vst dy.
\end{align} 
\end{subequations}

\subsection{A priori estimates}\label{subsec:kinetic-trajec}
We begin our study with a series of a priori estimates on the solution of the mean-field model 
(\ref{D-1}), and the growth rate of the $x$ and $v$-support of $f$. 
We first set
\[ 
\psi_0(\xi) := \xi, \quad \psi_i(\xi) := \xi_i \ \ i = 1,\ldots, d, \quad \mbox{ and } \quad \psi_{d+1}(\xi) := |\xi|^2. 
\]
Let $f$ be a classical solution to \eqref{D-1} with a rapid decay in phase space $\bbr^{2d}$.
A straightforward integration of \eqref{D-1} yields 
\begin{subequations}\label{eqs:mom}
\begin{eqnarray}\label{eq:moma}
\frac{d}{dt} \int_{\bbr^{2d}} \psi_i(v)
f(x,v) dv dx & = & \int_{\bbr^{2d}} \nabla_v \psi_i(v) \cdot \Col dv dx, \\
\frac{d}{dt} \int_{\bbr^{2d}} \psi_i(x)
f(x,v) dv dx & = & \int_{\bbr^{2d}} \nabla_{x} \big(\psi_i(x) \cdot v\big)
f(x,v) dv dx. \label{eq:momb}
\end{eqnarray}
\end{subequations}

\noindent
Using \eqref{eqs:mom} we obtain
\begin{proposition} Let $f$ be a classical solutions decaying fast enough at infinity in phase space.
Then following macroscopic quantities associated with $f$,
 satisfy
\begin{subequations}\label{eqs:moremom}
\begin{eqnarray}
& & \frac{d}{dt} \int_{\bbr^{2d}} v f(x,v) dx dv = 0; \label{eq:moremoma}\\
& & \frac{d}{dt} \int_{\bbr^{2d}} |v|^2 f(x,v) dxdv
= -\int_{\bbr^{4d}} r(x,y) |v-\vst|^2 f(x,v) f(y,\vst) d\vst dy dvdx;  \label{eq:moremomb}\\
& & \frac{d}{dt} \int_{\bbr^{2d}}
f^p(x,v) dv dx = -d(p-1) \int_{\bbr^{4d}} r(x,y) f(y,\vst) f^p(x,v) d\vst dv dy dx. 
\label{eq:moremomc}
\end{eqnarray}
\end{subequations}
\end{proposition}

\begin{proof}
 Equality \eqref{eq:moremoma} follows from \eqref{eq:moma} with $\psi_i(v) = v_i$,
\begin{eqnarray*}
  \int_{\bbr^{2d}} \nabla_v \psi_i(v) \cdot \Col dv dx 
  = \int_{\bbr^{4d}} r(x,y) (v_{*i}-v_i) f(y,\vst) f(x,v) d\vst dy dv dx = 0.
\end{eqnarray*}
The last integral vanishes due to antisymmetry of the integrand, realized by the interchange of variables $(x,v)\leftrightarrow (y,\vst)$. 
 The statement of \eqref{eq:moremomb} follows from \eqref{eq:momb} with $\psi_{d+1}(v)=|v|^2$,
 and observing that
\begin{eqnarray*}
 \displaystyle && \int_{\bbr^{2d}} \nabla_v \psi_{d+1}(v) \cdot \Col dv dx 
  = 2\int_{\bbr^{2d}} v \cdot \Col dv dx \\
 \displaystyle && \hspace{2cm} = 2\int_{\bbr^{4d}} r(x,y) v \cdot (\vst - v) f(y,\vst) f(x,v) d\vst dy dv dx \cr
 \displaystyle && \hspace{2cm} = -2 \int_{\bbr^{4d}} r(x,y) \vst \cdot (\vst - v) f(y,\vst) f(x,v) d\vst dy dv dx \cr
 \displaystyle && \hspace{2cm} = -\int_{\bbr^{4d}} r(x,y) |v-\vst|^2 f(x,v) f(y,\vst) d\vst dy dv dx.
\end{eqnarray*}

\noindent
Finally, we note the two identities, $f^{p-1} v\cdot\nablax f  \equiv \frac{1}{p} v\cdot\nablax  f^p$, and  
\[
\displaystyle f^{p-1} \nablav\cdot \Col \equiv \nablav\cdot \Big( \frac{\Qf  f^p}{p} \Big) + \Big(1-\frac{1}{p} \Big) \Big( \nablav\cdot \Qf  \Big) f^p, \qquad \Col=f\Qf.
\]
Integration of \eqref{D-1} against $f^{p-1}$ then yields
\begin{eqnarray*}
\frac{d}{dt} \int_{\bbr^{2d}} f^p dv dx  
 &=& -p \int_{\bbr^{2d}} f^{p-1} \Big( v\cdot\nablax f + \nablav\cdot\Col \Big) dv dx  \cr
 &=& -(p-1) \int_{\bbr^{2d}}  \Big( \nablav\cdot \Qf  \Big) f^p dv dx \cr
 &=& -d(p-1) \int_{\bbr^{4d}} r(x,y) f(y,\vst) f^p(x,v) d\vst dv dy dx.
\end{eqnarray*}
\end{proof}

Let $f$ be a classical kinetic solution of \eqref{D-1}. 
The statement of \eqref{eq:moremomc} shows that its $L^1(dxdv)$norm, the total macroscopic mass is conserved in time (while according to (\ref{eq:moremomc}), higher $L^p(dxdv)$-norms of $f$ decay in time), 
\begin{subequations}\label{eqs:Ms}
\begin{equation}\label{eq:Mi}
{\mathcal M}_0(t):=\int_{\bbr^{2d}}f(x,v,t)dxdv \equiv {\mathcal M}_0.
\end{equation}

Similarly, \eqref{eq:moremoma} tells us that the total macroscopic momentum is conserved in time,
\begin{equation}\label{eq:Mii}
{\mathcal M}_1(t) := \int_{\bbr^{2d}} v f(x,v,t) dxdv \equiv {\mathcal M}_1.
\end{equation}
Finally, \eqref{eq:moremomb} tells us that the total amount of macroscopic  energy is non-increasing in time,
\begin{equation}\label{eq:Miii}
\mathcal{M}_2(t):=\int_{\bbr^{2d}}|v|^2f(x,v,t)dvdx \leq \mathcal{M}_2(0).
\end{equation}
\end{subequations}
Here, ${\mathcal M}_0:={\mathcal M}_0(0), \ {\mathcal M}_1:={\mathcal M}_1(0)$ and $\mathcal{M}_2(0)$ denote, respectively, the initial amounts of mass, momentum and energy at $t=0$.
Next, we turn to the following a priori bound on the kinetic velocity.

\begin{lemma}\label{lem:4-2}
Let $[x(t), v(t)]$ be the particle trajectory issued from $(x,v) \in \mbox{supp}_{(x,v)} \fin $ at time $0$. Then the $i$-component of velocity trajectory, $v_i(t)
= v_i(t;0,x,v), \ i=1, \cdots, d$, satisfies
\begin{eqnarray*}
 \displaystyle &&  v_i(t) \in \Big(v_i(0) e^{-\lambda \rA \Iz t}  - \frac{\Jz}{\Iz} ( 1 - e^{-\lambda \rA \Iz t} ),  \ \  v_i(0) e^{-\lambda \Iz  \Phi(t)  }
+ \lambda \rA \Jz \int_0^t e^{-\lambda \Iz (\Phi(t)-\Phi(s))} ds \Big).
\end{eqnarray*}
Here, $\Phi(t):= \int_0^t \varphi(s)ds$,  $\Iz= \|\fin \|_{L^1_{x,v}}$ is the initial total mass and $\Jz:= \sqrt{{\mathcal M}_0{\mathcal M}_2}$.
\end{lemma}
\begin{proof}
For given $(x,v) \in \mbox{supp}_{x,v} \fin $, we set
$ x(s) \equiv x(s;0,x,v)$ and $v(s) \equiv v(s;0,x,v)$. 
Note that for each $i = 1, \cdots, d$, we have
\begin{align}\label{D-6}
& \Qvfi{i}{f(x(t),v(t),t)}  = \int_{\bbr^{2d}} r(x(t), y) (v_{*i}(t)- v_{i}) f(y,\vst)d\vst dy   \\
&\hspace{1cm}= v_i(t) + \int_{\bbr^{2d}} r(x(t), y) v_{*i} f(y,\vst) d\vst dy
-\Big( \int_{\bbr^{2d}} r(x(t), y) f(y,\vst) d\vst dy   \Big). \nonumber
\end{align}
Lower and upper bounds for the kinetic velocities are obtained in terms of the estimates,
\begin{align} 
\displaystyle  & \varphi(t) \Iz \leq 
\int_{\bbr^{2d}} r(x(t), y) f(y,\vst) d\vst dy    \leq \rA \Iz, \nonumber \\
\displaystyle &  \Big| \int_{\bbr^2} r(x,y) \vst f(y,\vst) d\vst dy \Big| \leq
\rA \sqrt{\|\fin \|_{L^1_{x,v}}} \sqrt{\| |v|^2 \fin \|_{L^{1}_{x,v}}} = \rA \Jz.
\nonumber
\end{align}
It then follows from (\ref{D-6}) that
\[
- \lambda \rA \Iz  v_i(s) -\lambda \rA \Jz \leq \frac{d}{dt} v_i(t) = \lambda 
\Qvfi{i}{f(x(t),v(t),t)} 
           \leq - \lambda \varphi(t)  \Iz v_i(t) + \lambda\rA \Jz,
\]
and the desired result follows by Gronwall's integration.
\end{proof}

\begin{remark}\label{rem:4-2}
Let $\Pt(t)$ denote the  $v$-projection  of $\mathrm{supp}f(\cdot,t)$,
\begin{equation}\label{eq:Pt}
\Pt(t) := \{ v \in \bbr^d: \exists~~(x,v) \in \bbr^{2d} \quad \mbox{ such that } \quad f(x,v,t) \not = 0 \}.
\end{equation}
Lemma \ref{lem:4-2} shows that if $\fin(x,\cdot)$ is compactly supported, then $\mathrm{supp} f(x,\cdot,t)$ remains finite, with  a weak growth estimate for the velocity trajectory,
\begin{equation}\label{eq:yyz}  
|v_i(t)| \leq  \max \Big \{ \eta_0+\frac{\Jz}{\Iz}, \eta_0+\lambda \Jz t  \Big \} 
\leq  \eta_0+ \frac{\Jz}{\Iz} + \lambda \Jz t, \quad \eta_0:=\max_{v\in \Pt(0)}|v|.
\end{equation}
\end{remark}

%
%
\subsection{Global existence of classical solutions}\label{subsec:kinetic-global}
In  this section  we develop  a global existence theory for
classical solutions of the Vlasov-type flocking equation, 

\begin{subequations}\label{eqs:E1-2}
\begin{align} \label{E1}
\displaystyle & \partial_t f + v\cdot\nablax  f + \lambda \nablav\cdot \big(f\Qf \big)=0, \quad x, v \in \bbr^d, t > 0, \\
\displaystyle \Qf (x,v,t) &= \int_{\bbr^{2d}} r(x,y) (\vst - v) f(y,\vst,t) d\vst dy, \quad r(x,y) = \frac{\rA}{(1 + |x-y|^2)^{\beta}},
\end{align}
subject to initial datum 
\begin{align}\label{E2}
f(x,v,0)  = \fin (x,v).
\end{align}
\end{subequations}

We begin by noting that the kinetic solution $f$ remains \emph{uniformly bounded}. 
To this end,  rewrite
the mean-field model \eqref{eqs:E1-2} in a 'non-conservative' form,
\begin{equation} \label{D-4}
\displaystyle \partial_t f + v \cdot \nabla_x f +  \lambda \Qf\cdot\nablav f 
 =  -\lambda f \nablav\cdot \Qf  ,
\quad x, v \in \bbr^d, t>0.
\end{equation}
Consider the particle trajectories, 
$[x(t), v(t)] \equiv [x(t;t_0,x_0,v_0), v(t;t_0,x_0,v_0)]$, passing through
$(x_0,v_0)\in \bbr^d \times \bbr^d$ at time $t_0\in \bbr_+$,
\begin{equation} \label{D-5}
\displaystyle \frac{d}{dt} x(t) = v(t), \qquad
\frac{d}{dt} v(t) = \lambda\Qvf{f(x(t),v(t),t)}.
\end{equation}
Noting that $-\nablav\cdot \Qf  =  d  \int_{\bbr^{2d}} r(x,y) f(y,\vst,t) d\vst dy$,  
we find
\[ 
\displaystyle \|\nablav\cdot \Qf \|_{L^{\infty}_{x,v}} \leq d \rA\|f\|_{L^1_{x,v}} = d \rA \Iz,  
\]
which implies that the following inequality holds along the particle trajectories,
\[
\displaystyle \frac{d}{dt} f(x(t),v(t),t) \leq \lambda d \rA \Iz f(x(t),v(t),t).
\]
It follows that as long as initial data $\fin$ has a finite mass,
there will be no finite time blow-up for $f(\cdot,t)$,
\begin{equation} \label{D-5-1}
\displaystyle \|f(t)\|_{L^{\infty}_{x,v}} \leq e^{\lambda d\rA
\Iz t} \|\fin \|_{L^{\infty}_{x,v}}.
\end{equation}
 
\noindent
Next, we turn to study the \emph{smoothness} of $f(\cdot,t)$.
Since the local existence theory will be followed from the standard fixed point argument, e.g.,  \cite{B-D-P}, we only obtain a priori $C^1$-norm bound of $f$ to conclude a global existence of classical solutions.

%
%
\begin{theorem}\label{thm:5-1}
Consider the flocking kinetic model (\ref{eqs:E1-2}). Suppose that the initial datum $\fin  \in
(C^1 \cap W^{1, \infty})(\bbr^{2d})$ satisfies
\begin{subequations}\label{eqs:fin}
\begin{enumerate}
\item Initial datum is compactly supported in the phase space, $supp_{(x,v)}\fin(\cdot)$ is bounded, and in particular, $\Pt(0) \subset B_{\eta_0}(0)$. 
\item Initial datum is $C^1$-regular and bounded:
\[ \displaystyle \sum_{0 \leq |\alpha|+ |\beta| \leq 1} \|\nabla_{x}^{\alpha} \nabla_{v}^{\beta} \fin \|_{L^{\infty}_{x,v}} < \infty. \]
\end{enumerate}
\end{subequations}
Then, for any $T \in (0, \infty)$, there exists a unique
classical solution $f \in C^1([0, T) \times \bbr^{2d})$.
\end{theorem}

\begin{proof} 
We  express the non-conservative kinetic model (\ref{D-4}) in terms of the nonlinear transport operator $\partial_t + v \cdot \nabla_x  + \lambda\Qf  \cdot \nabla_v$,
\begin{equation}\label{eq:non} 
{\mathcal T}f = -\Nvareps f\nablav\cdot \Qf, \qquad {\mathcal T}:= \partial_t + v \cdot \nabla_x  + \lambda\Qf  \cdot \nabla_v.
\end{equation}

We claim that there exist (possibly different)  positive constants, $C = C(d, \Nvareps,\Iz, \Jz) > 0$, such that 
\begin{subequations}\label{eqs:abc}
\begin{eqnarray}
\displaystyle & &|{\mathcal T}(f)| \leq C |f|, \label{eq:abc-a}\\
\displaystyle & & |{\mathcal T}(\partial_{x_i} f)| \leq C \Big( |f| +  (\eta(t) + 1) |\nabla_v f| + |\partial_{x_i} f|  \Big), \qquad \eta(t):=\max_{v\in \Pt(t)}|v|, \label{eq:abc-b}\\
\displaystyle & & |{\mathcal T}(\partial_{v_i} f)| \leq C \Big( |\partial_{x_i} f| + |\nabla_v f| \Big). \label{eq:abc-c}
\end{eqnarray}
\end{subequations}

To verify these inequalities, observe that by (\ref{eq:non})
\[
\displaystyle {\mathcal T}(f) = -\Nvareps f\nablav\cdot \Qf  
= \Nvareps d \int_{\bbr^{2d}} r(x,y) f(y,\vst,t)d\vst dy
\]
and (\ref{eq:abc-a}) follows with $C:= \Nvareps d \rA \Iz \geq \Nvareps\|\nablav\cdot\Qf (t)\|_{L^{\infty}_{x,v}}$.

\smallskip\noindent
Next, differentiating (\ref{eq:non}) we obtain
\[
\displaystyle {\mathcal T}(\partial_{x_i} f) =
-\Nvareps(\partial_{x_i} \Qf ) \cdot \nabla_v f - \Nvareps(\partial_{x_i}
\nablav\cdot \Qf ) f - \Nvareps(\nablav\cdot \Qf ) \partial_{x_i} f.
\]
Straightforward calculation yields
\[
\partial_{x_i} \Qf  =  -\int_{\bbr^{2d}} \frac{2 \beta\rA( x_i - y_i)}{ (1 + |x-y|^2)^{\beta +1}}(\vst-v) f(y,\vst,t) d\vst dy,
\]
and since the variation of the relevant kinetic velocities at time $t$ does not exceed 
$|v-\vst|\leq 2\eta(t)$, we find $\|\partial_{x_i} \Qf (t)\| \leq 4 \beta\rA  \eta(t) \Iz$;
similarly,
\[
\partial_{x_i} \nablav\cdot
\Qf  =  \int_{\bbr^{2d}} \frac{2 \beta\rA d ( x_i - y_i)}{ (1 + |x-y|^2)^{\beta +1}}f(y,\vst,t) d\vst dy \ \ \mapsto \ \ \|\partial_{x_i} \nablav\cdot\Qf (t)\|_{L^{\infty}_{x,v}} \leq 2 \beta d\rA  \Iz.
\]
We conclude that (\ref{eq:abc-b}) holds with, say,  $C= \Nvareps d\rA \mathcal{M}_0 (1+2\beta+ 4\beta\eta(t))$.

\smallskip\noindent
Finally, we differentiate (\ref{eq:non}) with respect to $v_i$ (noting that $\partial_{v_i} \nablav\cdot \Qf  = 0$)
\[
{\mathcal T}(\partial_{v_i} f) =
-\Nvareps\partial_{x_i} f - \Nvareps(\partial_{v_i} \Qf ) \cdot \nabla_v f 
-\Nvareps(\nablav\cdot \Qf )\partial_{v_i} f;
\]
Straightforward calculation then yields,
\[
\partial_{v_i} \Qf  = -\int_{\bbr^{2d}} r(x,y) f(y,\vst,t) d\vst dy \ \ 
\mapsto \ \ \|\partial_{v_i} \Qf (t)\|_{L^{\infty}_{x,v}} \leq \Iz,
\]
and (\ref{eq:abc-c}) follows with $C=\Nvareps+\Nvareps(d+1)\rA{\mathcal M}_0$.

Now, let ${\mathcal F}(t)$ measure the $W^{1,\infty}$-norm of $f(\cdot,t)$
\[
\displaystyle  {\mathcal F}(t) := \sum_{0 \leq |\alpha|+ |\beta| \leq 1} \|\nabla_{x}^{\alpha} \nabla_{v}^{\beta} f(t)\|_{L^{\infty}_{x,v}}. 
\]
The inequalities (\ref{eqs:abc}) imply 
\[ 
\displaystyle  \frac{d}{dt} {\mathcal F}(t) \lesssim   \big(\eta(t)+1\big)  {\mathcal F}(t). 
\]
Lemma \ref{lem:4-2} (see remark \ref{rem:4-2}), tells us that $\eta(t) \lesssim \eta_0 + t$, and we end up with the energy bound
\[ 
\displaystyle {\mathcal F}(t) \leq {\mathcal F}(0) e^{C(t + t^2)}, \qquad C=C(\eta_0,\Iz, \Jz,\beta, d, \rA). 
\]
Equipped  with this a priori $W^{1,\infty}$ estimate, standard continuation principle yields a global extension of local classical solutions.
\end{proof}

\begin{remarks}
\begin{enumerate}
\item The above a priori estimate need not be optimal. Since we used a rough estimate 
(\ref{eq:yyz}) for the size of $\Pt(t)$,
\[ 
\displaystyle    \max_{v \in \Pt(t)} |v| \lesssim \eta_0 + t, 
\]
we end with the quadratic exponential growth, $e^{C( t + t^2)}$. An optimal bound, however,  could  be $e^{Ct}$. Of course, one cannot expect a uniform bound for $C^1$-norm, because the one-particle distribution function may grow exponentially along  the particle trajectory (see \eqref{D-5-1}). 
\item The global existence of classical solution can be improved for more general kernels. 
\item For related works on kinetic granular type dissipative systems, we refer to \cite{Ja1, Ja2, J-P}.
\end{enumerate}
\end{remarks}
%
%
\section{Time-asymptotic behavior of kinetic flocking}\label{sec:kinetic-time}
\setcounter{equation}{0} 
In this section, we present the time-asymptotic flocking behavior of 
the kinetic model for flocking (\ref{eqs:E1-2}). As in the case with particle description  discussed in Section \ref{sec:CS}, we will show that the velocity of particles will contracted
 to the mean bulk velocity $u_c$, which corresponds to the velocity at the center of mass:
\[ 
u_c(t):=\frac{1}{\mathcal{M}_0}\int_{\bbr^{2d}}vf(x,v,t)dv dx, \qquad u_c(t) \equiv u_c(0). 
\]

We recall that the energy decay in  \eqref{eq:moremomb} 
\begin{equation}\label{eq:m2decay}
\frac{d}{dt}\mathcal{M}_2(t) \leq 0.
\end{equation}
We note that unlike granular flows, for example, e.g. \cite{B-D-P}, the energy decay 
(\ref{eq:m2decay}) does \emph{not} drive the energy to zero: if the initial momentum
${\mathcal M}_1 \not = 0$, then the kinetic energy $\mathcal{M}_2(t)$ has a nonzero lower bound, in analogy with the discrete case, consult remark \ref{rem:xyz}. Indeed, since the total mass and momentum are conserved,  ${\mathcal M}_i(t) \equiv {\mathcal M}_i, \, i=0,1$, (\ref{eq:moremomb}) implies
\begin{eqnarray*}
\displaystyle \frac{d}{dt} {\mathcal M}_2(t)  \geq  -\rA
\int_{\bbr^{2d}} |v-\vst|^2 f(x,v) f(y,\vst) d\vst dy dv dx 
                                    =  	-2\rA {\mathcal M}_0 {\mathcal M}_2(t) 
				        +2\rA|{\mathcal M}_1|^2,
\end{eqnarray*}
 and  Gronwall's lemma yields the following kinetic analog of  \eqref{eq:B5c},\eqref{eq:CSinq}
\begin{equation}\label{eq:B5mac}
\displaystyle {\mathcal M}_2(t) \geq {\mathcal M}_2(0) e^{-2\rA \Iz t} +
\frac{|{\mathcal M}_1|^2}{{\mathcal M}_0} ( 1 - e^{-2\rA\Iz t} ) \geq \frac{|{\mathcal M}_1|^2}{{\mathcal M}_0}.
\end{equation}

Thus, energy decay by itself does not assert flocking. As in the particle description, the emergence of the time-asymptotic flocking behavior depends on the  sufficiently slow decay rate of  the interparticle interaction $\varphi(s)=r(x(s),y(s))$. To this end,  
we let $\Rt$ denote the $x$-projection of $\mathrm{supp} f(\cdot,t)$
\begin{equation}\label{eq:Rt}
\Rt(t) := \{ x \in \bbr^d: \exists~~(x,v) \in \bbr^{2d} \quad \mbox{ such that } \quad f(x,v,t) \not = 0 \}, 
\end{equation}
and  denote its initial size, $\zeta_0:=\max_{x\in \Rt_0}|x|$.

\begin{lemma}\label{lem:Tdecay}
Let $f$ be a global classical solution to \eqref{eqs:E1-2}. Then, there exists a $\kappa_3 >0$ such that $\varphi(t)=\inf_{(x,y)\in \Rt(t)} r(x,y)$ satisfies
\begin{equation}\label{eq:Td} 
\displaystyle \varphi(t) \geq \rA\kappa_3^{-\beta} ( 1 + t^2 + t^4)^{-\beta}. 
\end{equation}
 The constant $\kappa_3$ is given by
\[ \displaystyle \kappa_3:= \max \Big \{ 1 + 12 \zeta^2_0, 12 \Big(\eta_0 + \frac{\Jz}{\Iz} \Big)^2,  3 \lambda^2 \Jz^2  \Big \}. 
\]
\end{lemma}

\begin{proof} Let $(x,v) \in \mbox{supp}_{x,v} \fin $. 
It follows from remark \ref{rem:4-2} that
\begin{equation}
\displaystyle   |v_i(t)|  \leq \eta_0 + \frac{\Jz}{\Iz} + \lambda \Jz t,
\end{equation}
and hence
\begin{eqnarray*}
\displaystyle |x_i(t;0,x,v)| &\leq& |x_i| + \int_0^t |v_i(s;0,x,v)| ds 
\leq \zeta_0 + \Big( \eta_0 +  \frac{\Jz}{\Iz} \Big) t + \frac{\lambda \Jz t^2}{2}.
\end{eqnarray*}
This gives an estimate on the size of $x$-support $\Rt(t)$ of $f$. For $x, y \in \Rt(t)$,
\begin{eqnarray*}
\zeta(t) \leq 1 + |x - y|^2 &\leq& 1 +  2 (|x|^2 + |y|^2) \cr
          &\leq& 1 +  4 \Big[ \zeta_0 + \Big( \eta_0 + \frac{\Jz}{\Iz} \Big) t + \frac{\lambda \Jz t^2}{2} \Big]^2 \cr
          &\leq& 1 + 12 \zeta_0^2 + 12 \Big(\eta_0 + \frac{\Jz}{\Iz}\Big)^2 t^2 + 3 \lambda^2 \Jz^2 t^4 \cr
          &\leq& \kappa_3 (1 + t^2 + t^4),
\end{eqnarray*}
and (\ref{eq:Td}) follows, $\varphi(t) \geq r(\zeta(t))$.
\end{proof}

Let $v-u_c$ denote the fluctuation (or peculiar) kinetic velocity. We will quantify the emergence of the time-asymptotic flocking behavior in term of the corresponding energy fluctuation 
\begin{equation} \label{TE1}
\displaystyle \Lambda[f(t)] := \int_{\bbr^{4d}} |v-u_c|^2 f(x,v,t)dv dx.
\end{equation}
The time-evolution estimate of $\Lambda[f(t)]$, will depends on the decay rate of $\varphi(t)$.

Let $f$ be a classical solution of  \eqref{eqs:E1-2} with compact support in $x$ and $v$. Direct calculation implies
\begin{align}
 \frac{d}{dt} \Lambda[f(t)] &= \int_{\bbr^{2d}} |v - u_c|^2 \partial_t f(x,v) dv dx \cr
&= -\int_{\bbr^{2d}} |v - u_c|^2 v\cdot\nablax f dv dx 
- \lambda \int_{\bbr^{2d}} |v - u_c|^2 \nablav\cdot \Col dv dx=:\mathcal{I}_1+\mathcal{I}_2. \nonumber 
\end{align}
The first term on the right vanishes by the divergence theorem 
\[
\displaystyle  \mathcal{I}_1= -\int_{\bbr^{2d}} |v - u_c|^2 v\cdot\nablax f dv dx= -\int_{\bbr^{2d}} \mbox{div}_x \Big( |v - u_c|^2 v f \Big) dv dx = 0.
\]
The second  term is simplified as follows.
\begin{align}
\displaystyle &  \mathcal{I}_2= 2 \lambda \int_{\bbr^{2d}}  (v - u_c) \cdot \Col  dv dx \nonumber \\
\displaystyle & \hspace{2cm} = -2 \lambda \int r(x, y) (v - u_c) \cdot (v - \vst) f(y,\vst) f(x,v) d\vst dv dy dx
 \nonumber \\
\displaystyle & \hspace{2cm} = -2 \lambda \int r(x,y) v \cdot (v - \vst) f(y,\vst) f(x,v) d\vst dv dy dx \nonumber \\
\displaystyle & \hspace{2cm} = - \lambda \int r(x, y) |v - \vst|^2 f(y,\vst) f(x,v)d\vst dv dy dx. \nonumber
\end{align}
We summarize the  last three equalities in the following lemma.

\begin{lemma}\label{lem:yzt}
Let $f$ be a classical solution of (\ref{eqs:E1-2}) subject to compactly supported initial conditions $\fin$. Then, the decay of the energy  functional $\Lambda[f(t)]$ in (\ref{TE1}) is governed by  
\begin{equation}\label{eq:yzt}
\frac{d}{dt} \Lambda[f(t)] = - \lambda \int r(x, y) |v - \vst|^2 f(y,\vst) f(x,v)d\vst dv dy dx.
\end{equation}
\end{lemma}

Equipped with lemma \ref{lem:yzt} we can state the main result of this section.
\begin{theorem}\label{thm:asy}
Let $f$ be the classical kinetic solution constructed in Theorem \ref{thm:5-1}. Then,  the decay of its energy fluctuations around the mean bulk velocity $u_c$, is given by 

\begin{equation}\label{eq:asy}
\Lambda[f(t)] \lesssim \Lambda[\fin ]\times
\left\{%
\begin{array}{ll}
   C_3 e^{-\kappa_4 t^{1-4\beta}}, & 0 \leq \beta < \frac{1}{4}, \\ \\
    (1+t)^{-\kappa_5}, & \beta = \frac{1}{4}. \\
\end{array}%
\right. 
\end{equation}
The constants involved are $\kappa_4=1/(3\kappa_3)^\beta(1-4\beta)>0$ and $\kappa_5=2\lambda\rA/\sqrt[4]{3\kappa_3}>0$.
\end{theorem}
\begin{proof}
Lemma \ref{lem:yzt} implies
\begin{align}
\frac{d}{dt} \Lambda[f(t)] & =  - \lambda \int_{\bbr^{4d}} r(x, y) |v - \vst|^2 f(y,\vst) f(x,v)d\vst dv dy dx \nonumber \\
& \leq -  \lambda \varphi(t) \int_{\bbr^{4d}} |v - \vst|^2 f(y,\vst) f(x,v)d\vst dv dy dx. 
  = -2 \lambda \varphi(t) \Iz \Lambda[f(t)]. \nonumber
\end{align}
As before, the identity $|v-\vst|^2= |v-u_c|^2 +|\vst-u_c|^2 + 2(v-u_c)\cdot(\vst-u_c)$ induces the corresponding decomposition of the integrand on the right. Noting that 
\[ 
\displaystyle \int_{\bbr^{2d}} (v - u_c) f(x,v,t) dv dx = 0. 
\]
We conclude
\[
\frac{d}{dt} \Lambda[f(t)] \leq  -2 \lambda \varphi(t) \Iz \Lambda[f(t)]
\]
and Gronwall's integration yields
\begin{equation} \label{E5}
\displaystyle  \Lambda[f(t)] \leq  \Lambda[\fin ] e^{- 2 \lambda \Phi(t)}, \qquad 
\Phi(t)=\int^t_0\varphi(s)ds.
\end{equation}
We distinguish between two cases.

\medskip\noindent
{\bf Case 1 }[$0 \leq \beta < \frac{1}{4}$].  According to  Lemma \ref{lem:Tdecay} 
\[ 
\varphi(t) \geq \rA \kappa_3^{-\beta} (1 + t^2 + t^4)^{-\beta} \geq \rA(3 \kappa_4)^{-\beta} t ^{-4 \beta}, \qquad t \geq 1. 
\]
We compute for $t \geq 1$,
\[ 
\displaystyle -\Phi(t) \lesssim  - \int_1^t \varphi(\tau) d\tau \leq - \frac{\rA}{(3 \kappa_4)^{\beta} (1- 4 \beta)} (t^{1-4\beta} - 1)
\]
which yields the first part of (\ref{eq:asy}).

\medskip\noindent
{\bf Case 2 }[$\beta = \frac{1}{4}$]. For $t \geq 1$ we have
\[ 
\displaystyle \exp \Big( -2 \lambda \rA \Phi(t) \Big) \leq  \exp \Big(- \frac{2 \lambda \ln t}{(3 \kappa_4)^{1/4}} \Big) \lesssim 
(1+t)^{-\frac{2\lambda\rA}{\sqrt[4]{3 \kappa_4}}}, 
\]
and the second part of (\ref{eq:asy}) follows.
\end{proof}

%
%
\section{From kinetic to hydrodynamic description of flocking}\label{sec:hydro}
\setcounter{equation}{0}
In this section we discuss  the hydrodynamic description for flocking, which is formally obtained by taking moments of the kinetic model (\ref{eqs:E1-2}) . 

\begin{subequations}\label{ap-1}
\begin{align}
\displaystyle & \partial_t f + v\cdot\nablax f)+\lambda \nablav\cdot \Col=0, \quad x, v \in \bbr^d,~~ t > 0, \\
\displaystyle & \Col(x,v,t) = \int_{\bbr^{2d}} r(x,y) (\vst-v)f(x,v,t)f(y,\vst,t) d\vst dy.
\end{align}
\end{subequations}
We first set hydrodynamic variables: the mass $\rho:= \int_{\bbr^d} f d\vxi$, the 
momentum, $\rho u := \int_{\bbr^d} \vxi f d\vxi$, and the energy,
$\rho E := 1/2\int_{\bbr^d} {|\vxi|^2} f d\vxi$, which is the sum of kinetic and internal energies (corresponding to the first two terms in the decomposition of kinetic velocities $|v|^2=|u|^2+|v-u|^2+2(v-u)\cdot u$)
\begin{equation}\label{eq:En} 
\rho E = \rho e + \frac{1}{2}\rho{|u|^2}, \qquad \rho e:=\frac{1}{2}\int |v-u(x)|^2f(x,v,t)dv.
\end{equation}
For notational simplicity, we suppress time-dependence, denoting
$\rho(x) \equiv \rho(x,t), \ u(x) \equiv u(x,t)$ and $E(x) = E(x,t)$ when the context is clear.

We compute the $v$-moments of (\ref{ap-1}): multiply \eqref{ap-1} against $1, \vxi$ and ${|\vxi|^2}/{2}$  and integrate over the velocity space $\bbr^d$. We end up with the system of equations,
\begin{subequations}\label{ap-2}
\begin{align}
\partial_t \rho + \nabla_x  \cdot (\rho u) &= 0, \label{eq:Mass}\\
\displaystyle \partial_t (\rho u) + \nabla_x \cdot \Big(\rho u \otimes u  + P \Big) &= {\mathcal S}^{(1)}, \label{eq:Momentum}\\
\displaystyle \partial_t (\rho E) + \nabla_x \cdot \Big(\rho E u + P u + q  \Big) &= {\mathcal S}^{(2)}. \label{eq:Energy}
\end{align}
Here, ${\mathcal S}^{(j)}, j=1,2$, are  the nonlocal source terms given by
\begin{eqnarray}
\displaystyle {\mathcal S}^{(1)}(x,t) &:=& -\lambda \int_{\bbr^{d}} r(x,y) (u(x) - u(y)) \rho(x) \rho(y) dy, \label{eq:Sone}\\
\displaystyle {\mathcal S}^{(2)}(x,t) &:=& -\lambda \int_{\bbr^d} r(x,y) \Big[ E(x) + E(y) - u(x) \cdot u(y) \Big] \rho(x) \rho(y) dy, \label{eq:Stwo}
\end{eqnarray}
and $P = (p_{ij}), q= (q_i)$ denote, respectively, the stress
tensor and heat flux vector,
\begin{equation}\label{eq:pandq}
p_{ij} := \int_{\bbr^d} (\vxi_i - u_i) (\vxi_j - u_j) f dv,  \quad q_i
:= \int_{\bbr^d} (\vxi_i - u_i) |\vxi-u|^2 f dv.
\end{equation}
\end{subequations}

\begin{remark}\label{rem:5-1}
The total mass of the source term ${\mathcal S}^{(1)}$ vanishes: exchange of variables
$x\leftrightarrow y$ yields
\begin{subequations}\label{eqs:5-1}
\begin{align}\label{eq:5-1a}
\displaystyle & & \int_{\bbr^d} {\mathcal S}^{(1)}(x,t) dx = 
-\lambda\int_{\bbr^{2d}} r(x,y) (u(x)-u(y))  \rho(x) \rho(y) dxdy = 0. 
\end{align}
The source term ${\mathcal S}^{(2)}$ is non-positive: using (\ref{eq:En}) we find
\begin{align}\label{eq:5-1b}
\displaystyle && {\mathcal S}^{(2)}(t) = -\lambda \int_{\bbr^d} r(x,y) \Big[  \frac{1}{2} |u(x) - u(y)|^2 + e(x) + e(y) \Big] \rho(x) \rho(y) dy \leq 0.
\end{align}
\end{subequations}
\end{remark}

We conclude that the total mass and momentum, $\int\rho(x,t)dx$ and $\int\rho(x,t)u(x,t)dx$, are conserved in time. The total energy, however, 
\[
{\mathcal E}(t)=\int_{\bbr^d}\rho(x,t)E(x,t)dc = \frac{1}{2}\int_{\bbr^{2d}}|v|^2f(x,v,t) dvdx,
\]
is dissipating, which is responsible for the formation of time-asymptotic flocking behavior. We turn to quantify this decay.
We first write  the total energy as the sum of total kinetic and potential energies, corresponding to (\ref{eq:En}),
\[ 
{\mathcal E}(t) = {\mathcal E}_{\mbox{k}}(t) + {\mathcal E}_{\mbox{p}}(t), \quad 
{\mathcal E}_{\mbox{k}}(t):= \frac{1}{2}\int_{\bbr^d} \rho(x,t) |u(x,t)|^2 dx, \quad 
{\mathcal E}_{\mbox{p}}(t) := \frac{1}{2}\int_{\bbr^{2d}} |\vxi - u(x,t)|^2 f d\vxi dx. 
\]

\begin{lemma}\label{lem:fabc}
The time evolution of the total, kinetic and internal energies is governed by 
\begin{subequations}
\begin{eqnarray}
  && \ \ \frac{d}{dt} {\mathcal E}(t) = -\Nvareps \int_{\bbr^{2d}} r(x,y)\Big[\frac{1}{2}|u(x) - u(y)|^2+e(x)+e(y)\Big] \rho(x) \rho(y) dy dx; \label{eq:fabci}\\
   && \ \ \frac{d}{dt}  {\mathcal E}_{\mbox{p}}(t) = 
-\lambda \int_{\bbr^{2d}} \!\!\!\!\!r(x,y) \big(e(x)+e(y)-u(x)\cdot u(y)\big) \rho(x) \rho(y) dy dx -2 \int_{\bbr^d} (\nabla_x \cdot u) \rho e dx; \label{eq:fabcii}\\
 && \ \ \frac{d}{dt}  {\mathcal E}_{\mbox{k}}(t) = -\frac{\lambda}{2} \int_{\bbr^{2d}} r(x,y) \Big(|u(x)|^2+|u(y)|^2\Big) \rho(x) \rho(y) dx dy + 2 \int_{\bbr^{d}} (\nabla_x \cdot u) \rho e dx.\label{eq:fabciii}
\end{eqnarray}
\end{subequations}
\end{lemma}
\begin{proof}
The equality  (\ref{eq:fabci}) follows  from 
integration of (\ref{eq:Energy}) and invoking (\ref{eq:5-1b}).\newline
For the decay rate of the total internal energy ${\mathcal E}_{\mbox{p}}(t)$ in (\ref{eq:fabcii}), we  calculate
\begin{eqnarray*}
\frac{d}{dt} {\mathcal E}_{\mbox{p}}(t) &=& \int_{\bbr^{2d}} \partial_t \Big( \frac{|\vxi-u|^2}{2} \Big) f dv dx + \int_{\bbr^{2d}}
 \frac{|\vxi-u|^2}{2} \partial_t f dv dx =: {\mathcal K}_1 + {\mathcal K}_2.
\end{eqnarray*}
We estimate ${\mathcal K}_i$ separately. 
The first term, ${\mathcal K}_1$ vanishes, for
\begin{eqnarray*}
{\mathcal K}_1 &=& \int_{\bbr^{2d}} (u -\vxi) \cdot \partial_t u f d\vxi dx 
= \int_{\bbr^{2d}} u \cdot \partial_t u f d\vxi dx - \int_{\bbr^{2d}} \partial_t u \cdot (\vxi f) d\vxi dx \cr
& = &  \int_{\bbr^{d}} u \cdot \partial_t u \rho dx - \int_{\bbr^{d}} \partial_t u \cdot (\rho u) dx = 0.
\end{eqnarray*}
For the second term, ${\mathcal K}_2$, we use the kinetic model to find
\begin{eqnarray*}
{\mathcal K}_2 &=& \frac{1}{2} \int_{\bbr^{2d}} |\vxi-u|^2 \partial_t f d\vxi dx 
               = -\frac{1}{2} \int_{\bbr^{2d}} |\vxi-u|^2 \Big(  \vxi\cdot\nablax f +  \lambda \nablav\cdot \Col \Big) d\vxi dx \cr
               &=& -\frac{1}{2} \int_{\bbr^{2d}} |\vxi -u|^2 \vxi\cdot\nablax f d\vxi dx 
		   +{\lambda} \int_{\bbr^{2d}} (\vxi-u)\cdot\Col d\vxi dx 
		=: {\mathcal K}_{21} + \lambda {\mathcal K}_{22}.
\end{eqnarray*}

The term ${\mathcal K}_{21}$ amounts to
\begin{eqnarray*}
{\mathcal K}_{21} &=& -\frac{1}{2}\int_{\bbr^{2d}} |\vxi-u|^2 \vxi\cdot\nablax f d\vxi dx =\frac{1}{2}\int_{\bbr^{2d}} \Big( \nabla_x |\vxi-u|^2 \Big) \cdot (\vxi f) d\vxi dx \cr
   &=& - \int_{\bbr^{2d}} (\vxi - u) \cdot \Big( (\nabla_x \cdot u) \vxi f \Big) d\vxi dx  = 
-2 \int_{\bbr^d} (\nabla_x \cdot u) (\rho e) dx.
\end{eqnarray*}
A lengthy calculation shows that the remaining term,  ${\mathcal K}_{22}$, equals 
\begin{eqnarray*}
\displaystyle {\mathcal K}_{22} &=&  \int_{\bbr^{4d}} r(x,y) (\vxi - u(x)) \cdot (\vst -v) f(x,\vxi) f(y,\vxi_*) d\vxi_* d\vxi dy dx \cr
\displaystyle &=&  \int_{\bbr^{4d}} r(x,y) \vxi \cdot (\vst -v) f(x,\vxi) f(y,\vxi_*) d\vxi_* d\vxi dy dx  \cr
\displaystyle && - \int_{\bbr^{4d}} r(x,y) u(x) \cdot (\vst-v) f(x,\vxi) f(y,\vxi_*) d\vxi_* d\vxi dy dx \cr
\displaystyle &=& -\frac{1}{2}\int_{\bbr^{4d}} r(x,y) |\vxi - \vxi_*|^2 f(x,\vxi) f(y,\vxi_*) d\vxi_* d\vxi dy dx \cr
\displaystyle && +\frac{1}{2}\int_{\bbr^{2d}} r(x,y) \Big(|u(x)-u(y)|^2\Big) \rho(x) \rho(y) dy dx \cr
\displaystyle &=&  -\int_{\bbr^{2d}} r(x,y) \Big(E(x)+E(y)-\frac{1}{2}|u(x)-u(y)|^2\Big) \rho(x) \rho(y) dy dx \cr
\displaystyle &=& -\int_{\bbr^{2d}} r(x,y) \Big(e(x)+e(y)+u(x)\cdot u(y)\Big) \rho(x) \rho(y) dy dx
\end{eqnarray*}
Finally (\ref{eq:fabciii}) follows by subtracting  (\ref{eq:fabcii}) from (\ref{eq:fabci}).
\end{proof}

Next, we present a fundamental estimate for the flocking behavior to the system \eqref{ap-2}. We set
\begin{equation} \label{ap-3}
\displaystyle {\Gamma}(t)  := \int_{\bbr^{2d}}\Big(\frac{1}{2}|u(x) - u(y)|^2 + e(x) + e(y) \Big) \rho(x) \rho(y) dy dx.
\end{equation}
The functional ${\Gamma}(t)$ can be expressed in terms of the moments $\mathcal{M}_i$ in (\ref{eqs:Ms}) (corresponding to the splitting of its integrand $\frac{1}{2}|u(x) - u(y)|^2 + e(x) + e(y) \equiv \rho(x)E(x)+ \rho(y)E(y) -u(x)\cdot u(y)$),
\[ 
\displaystyle {\Gamma}(t) \equiv 2 {\mathcal E}(t) \Iz - |{\mathcal M}_1|^2, \qquad 
\mathcal{E}(t)={\mathcal M}_1(t)=2\mathcal{E}(t). 
\]
Since ${\mathcal M}_i, i=0,1$ are constants, this reveals that ${\Gamma}(t)$ is essentially the total energy. We arrive at the main theorem of this section.

\begin{theorem}\label{thm:fundamental}
Assume  $(\rho, u, e)$ is a  smooth solution of the system \eqref{ap-2},  $(\rho, u, e)\in C^1([0, T) \times \bbr^{d})$. Then we have
\[
\displaystyle   {\Gamma}(t) \leq {\Gamma}(0) e^{-2 \Iz  \lambda \Phi(t)}, \qquad \Phi(t)=\int_0^t\varphi(s)ds, \ \ \varphi(s):=\inf_{(x_0,y_0)} \, r\big(x(s),y(s)\big)ds. 
\]
Here, $\Iz$ is the initial total mass, $\Iz= \|\rho_0\|_{L^1}$ and the infimum is taken over all particle  trajectories, $(x_0,y_0)\mapsto (x(s),y(s))$.
\end{theorem}

\begin{proof} We use (\ref{eq:fabci}) and the relation ${\Gamma}(t) =2{\mathcal E}(t) \Iz - 
|{\mathcal M}_1|^2$ to find
\begin{eqnarray}
\displaystyle  \frac{d}{dt} {\Gamma}(t) &=& -   \lambda  \Iz \int_{\bbr^{2d}} r(x,y) \Big( |u(x) - u(y)|^2 + 2 e(x) + 2e(y) \Big)
                                           \rho(x) \rho(y) dy dx \label{eq:ed}\\
                                       &\leq& -   \lambda \Iz\varphi(t) \int_{\bbr^{2d}} \Big( |u(x) - u(y)|^2 + 2 e(x) + 2e(y) \Big)
                                           \rho(x) \rho(y) dy dx \nonumber\\
                                       &=& -2   \lambda \Iz \varphi(t) \Gamma(t). \nonumber
\end{eqnarray}
Gronwall's inequality then yields the desired result.
\end{proof}

We  conclude that whenever
the interparticle interaction, $\varphi(s)=\inf r(x(s),y(s))$ decays \emph{slowly} enough so that its primitive, $\Phi(t)$, diverges, then flocking occurs, $|u(x,t)-u(y,t)| \rightarrow 0$, in agreement with the flocking behavior of the C-S particle model, 
consult remark \ref{rem:2-1}.  
It is remarkable that the emergence of flocking is deduced here \emph{independently} of the constitutive relation for $P$.
In this context we observe that  energy dissipation,
driven by the negative source term $\mathcal{S}^{(2)}$ in (\ref{eq:Energy}) vanishes as $t\rightarrow \infty$. Indeed, theorem \ref{thm:fundamental} tells us that by (\ref{eq:5-1b}),
\[
 |\mathcal{S}^{(2)}| \leq A \Gamma(t) \rightarrow 0.
\] 

\section{Epilogue: flocking dissipation and entropy}
We have seen that the self-propelled flocking dynamics is driven by energy dissipation. 
The dissipation mechanism reveals itself through energy decay in the particle description 
(\ref{eq:2.4ca}),  in the kinetic description (\ref{eq:moremomb}) and equivalently, in the hydrodynamic description (\ref{eq:ed}). Observe that,
\begin{equation}\label{eq:6}
\Gamma(t) = {\mathcal{M}_2}{\mathcal {M}_0}  -|\mathcal{M}_1|^2 \geq 0.
\end{equation}
The right of (\ref{eq:6}) is the usual Cauchy-Schwartz inequality
\[
\Big|\int_{\bbr^{2d}} vf(x,v)dvdx\Big|^2 \leq \int_{\bbr^{2d}} |v|^2f(x,v)dvdx \times \int_{\bbr^{2d}}f(x,v)dvdx.
\] 
Thus, by theorem \ref{thm:fundamental}, $\Phi(t)\rightarrow \infty$ implies time asymptotic flocking by letting $\Gamma(t) \rightarrow 0$ which in turn, enforces an approximate Cauchy-Schwartz \emph{equality}. It then follows that $v$ approaches the bulk velocity, $v \rightarrow u_c$ as $t \rightarrow \infty$.
We refer to this mechanism as flocking dissipation. It is intimately related to the entropy \emph{increase} in the kinetic model (\ref{D-1}). To this end we compute
\begin{eqnarray*} 
\frac{d}{dt}\int_{\bbr^{2d}} f\log(f) dx dv &=& -\Nvareps \int_{\bbr^{4d}}r(x,y)\nabla_v \log(f)\cdot(v-\vst)ff_* d\vst dvdydx \\
&=& -\Nvareps \int_{\bbr^{4d}}r(x,y)\nabla_v f\cdot(v-\vst)f_* d\vst dvdydx \\
&=&  \Nvareps \int_{\bbr^{4d}}r(x,y)ff_* d\vst dvdydx =\Nvareps \int_{\bbr^{4d}}r(x,y)\rho(x)\rho(y) dydx.
\end{eqnarray*}
This is a reversed $H$-theorem. Entropy increases due to the ``improbable" statistical behavior of particles with increasingly highly correlated velocities, as they flock towards particle-independent bulk velocity.

\end{document}